\documentclass[11pt,a4paper]{article}
\usepackage[utf8]{inputenc}
\usepackage[T1]{fontenc}
\usepackage[french,english]{babel}
\usepackage{xcolor}
\usepackage{amsmath}
\usepackage{amssymb}
\usepackage{amsfonts}
\usepackage{amsthm}
\usepackage{mathtools}
\usepackage{mathrsfs}
\usepackage{bm}
\usepackage{tensor}
\usepackage[a4paper,margin=2.7cm]{geometry}
\usepackage{microtype}
\usepackage{setspace}
\usepackage[
    colorlinks=true,
    linkcolor=blue!60!black,
    citecolor=blue!60!black,
    urlcolor=blue!60!black
]{hyperref}

\usepackage[nameinlink,noabbrev]{cleveref}

\usepackage{booktabs}
\usepackage{enumitem}
\newtheorem{theorem}{Theorem}[section]

\newtheorem{corollary}[theorem]{Corollary}
\theoremstyle{definition}
\newtheorem{definition}[theorem]{Definition}
\newtheorem{remark}[theorem]{Remark}

\newcommand{\R}{\mathbb{R}}

\newcommand{\Ric}{\operatorname{Ric}}

\newcommand{\Sol}{\mathrm{Sol}_3}

\newcommand{\doi}[1]{%
    \href{https://doi.org/#1}{\texttt{#1}}%
}
\title{\textbf{Ricci-Yamabe solitons on the Lie group $\Sol \times \mathbb{R}^n$}}
\author{
Abdou Bousso
\thanks{
Département de Mathématiques et Informatique,
Université Cheikh Anta Diop de Dakar, Sénégal.\\
E-mail:
\href{mailto:abdoukskbousso@gmail.com}
{abdoukskbousso@gmail.com}
}
\and
Ameth Ndiaye
\thanks{
Département de Mathématiques, FASTEF,
Université Cheikh Anta Diop de Dakar, Sénégal.\\
E-mail:
\href{mailto:ameth1.ndiaye@ucad.edu.sn}
{ameth1.ndiaye@ucad.edu.sn}
}
}

\date{}

\begin{document}

\maketitle
\begin{abstract}
In this article, we study Ricci-Yamabe solitons on the Lie group $\mathrm{Sol} \times \mathbb{R}^n$ equipped with a natural left-invariant Riemannian metric, explicitly determining the vector fields that characterize them. We then deduce that, in the case of a Ricci soliton, it is expanding, whereas in the case of a Yamabe soliton, it is shrinking. Finally, we show that if this Lie group is a gradient Ricci-Yamabe soliton, the vector field belongs to $\operatorname{Span}\{\partial_{t_1}, \dots, \partial_{t_n}\}$, and we explicitly provide the Perelman potential.
\vspace{0.3cm}


\end{abstract}

\section{Introduction}
Geometric solitons are a central object in Riemannian geometry, extending the notion of Einstein metrics \cite{Besse, Petersen}. Among variational extensions, Ricci-Yamabe solitons unify Ricci and Yamabe flows \cite{Hamilton}. A Riemannian manifold $(M, g)$ admits a Ricci-Yamabe soliton if there exist real constants $\lambda$, $\beta_1$, $\beta_2$ and a smooth vector field $X$ on $M$ satisfying the tensor equation:
\begin{equation}
\beta_1 \operatorname{Ric} + \frac{1}{2}\mathcal{L}_X g = (\lambda + \beta_2 \operatorname{Scal})g,
\label{eq:RY-def}
\end{equation}
where $\operatorname{Ric}$ denotes the Ricci tensor, $\operatorname{Scal}$ the scalar curvature, and $\mathcal{L}_X g$ the Lie derivative of the metric along $X$. It is denoted by $(M,g,X,\beta_1,\lambda,\beta_2)$.

Much recent work has been devoted to the study of these structures and their generalizations on various geometric spaces and Lie groups, particularly in hyperbolic spaces and their products \cite{BoussoNdiayeJDSGT2025, BoussoNdiaye2026Hn, BoussoNdiayeH2R2025, DiopBoussoNdiayeMandal2026,  BoussoNdiayeSol3}. This article extends the study conducted on $\Sol$ \cite{BoussoNdiayeSol3} to the direct product $M = \Sol\times\R^n$ (see also \cite{doCarmo, Lee, ONeill} for the foundations of Riemannian geometry), providing a detailed and rigorous calculation of the underlying geometric objects.

\section{Preliminaries}
Consider the Lie group $M = \mathrm{Sol}_3 \times \mathbb{R}^n$ of dimension $3+n$, equipped with the left-invariant Riemannian metric:
\begin{equation}
g = e^{2z}\mathrm{d}x^2 + e^{-2z}\mathrm{d}y^2 + \mathrm{d}z^2 + \sum_{k=1}^n \mathrm{d}t_k^2,
\label{eq:metric}
\end{equation}
in local coordinates $(x, y, z, t_1, \dots, t_n) \in \R^{3+n}$.

Let us introduce the orthonormal basis of vector fields $\{E_1, E_2, E_3, F_1, \dots, F_n\}$ defined by:
\begin{equation}
E_1 = e^{-z}\partial_x, \quad E_2 = e^z\partial_y, \quad E_3 = \partial_z, \quad F_k = \partial_{t_k} \quad (\text{for } k = 1, \dots, n).
\end{equation}
The non-zero Lie brackets characterizing the underlying Lie algebra are:
\begin{equation}
[E_3, E_1] = -E_1, \quad [E_3, E_2] = E_2,
\end{equation}
with all other brackets between basis elements being zero, as the fields $F_k$ are central.

Using Koszul's formula for the Levi-Civita connection $\nabla$,
\begin{equation}
2g(\nabla_U V, W) = g([U,V], W) - g([V,W], U) + g([W,U], V),
\end{equation}
let us determine the non-zero connections: \begin{itemize}
    \item for $\nabla_{E_1} E_1$ : with $W = E_3$, $2g(\nabla_{E_1}E_1, E_3) = -g(E_1,[E_1,E_3]) + g(E_1,[E_3,E_1]) = -(-1) + (-1) = -2$, hense $\nabla_{E_1} E_1 = -E_3$.
    \item for $\nabla_{E_1} E_3$ : with $W = E_1$, $2g(\nabla_{E_1}E_3, E_1) = -g(E_1,[E_3,E_1]) + g(E_1,[E_1,E_3]) = -(-1) + 1 = 2$, hense $\nabla_{E_1} E_3 = E_1$.
    \item By symmetry, $\nabla_{E_2} E_2 = E_3$ and $\nabla_{E_2} E_3 = -E_2$.
\end{itemize}

All other connections involving fields $F_k$ are zero due to the flatness of the factor $\mathbb{R}^n$.

The Ricci tensor is defined by $$\Ric(U, V) = \sum_{j=1}^n g(R(E_j, U)V, E_j), \quad \text{where}\quad R(U,V)W = \nabla_U \nabla_V W - \nabla_V \nabla_U W - \nabla_{[U,V]}W.$$

We therefore have:
   $$\operatorname{Ric}(E_1, E_1) = g(R(E_2, E_1)E_1, E_2) + g(R(E_3, E_1)E_1, E_3) = 1 + (-1) = 0;$$
   $$\operatorname{Ric}(E_2, E_2) = g(R(E_1, E_2)E_2, E_1) + g(R(E_3, E_2)E_2, E_3) = 1 + (-1) = 0;$$
   $$\operatorname{Ric}(E_3, E_3) = g(R(E_1, E_3)E_3, E_1) + g(R(E_2, E_3)E_3, E_2) = -1 + (-1) = -2;$$
   $$\operatorname{Ric}(F_k, F_k) = 0 \quad (\text{for all } k = 1, \dots, n).$$
The scalar curvature $\operatorname{Scal}$ of the manifold of dimension $3+n$ is obtained by contraction:
\begin{equation}
\operatorname{Scal} = \sum_{i=1}^3 g(E_i, E_i)\operatorname{Ric}(E_i, E_i) + \sum_{k=1}^n g(F_k, F_k)\operatorname{Ric}(F_k, F_k) = 0 + 0 + (-2) + 0 = -2.
\end{equation}
\begin{definition}
Let $(M, g)$ be a Riemannian manifold of dimension $m$, $X$ a smooth vector field, and $\lambda, \beta_1, \beta_2$ real constants. The sextuple $(M, g, X, \beta_1, \lambda, \beta_2)$ is called a \textbf{Ricci-Yamabe soliton} if equation \eqref{eq:RY-def} is satisfied. Furthermore, it is called a \textbf{gradient} soliton if there exists a smooth function $f$ on $M$ such that $X = \nabla f$, with the Lie derivative then expressed as $\frac{1}{2}\mathcal{L}_{\nabla f} g = \mathrm{Hess}\, f$.
\end{definition}

\begin{remark}
The quadruple $(M, g, X, \lambda)$ corresponds to a \textbf{Ricci soliton} (respectively a \textbf{Yamabe soliton}) when $(\beta_1, \beta_2) = (1, 0)$ (respectively $(\beta_1, \beta_2) = (0, 1)$) in equation \eqref{eq:RY-def}. 
Furthermore, a Ricci-Yamabe soliton (as well as a Ricci or Yamabe soliton) is termed:
\begin{itemize}
    \item \textbf{steady} if $\lambda = 0$;
    \item \textbf{shrinking} if $\lambda > 0$;
    \item \textbf{expanding} if $\lambda < 0$.
\end{itemize}
\end{remark}
Let us now choose a vector field  $X = \sum\limits_{i=1}^3 h_iE_i+\sum\limits_{k=1}^n f_k F_k$ in $\mathfrak{X}(\Sol\times \R^n)$. The Lie derivative of the metric is given by:
\begin{equation}
(\mathcal{L}_X g)(Y, Z) = g(\nabla_Y X, Z) + g(Y, \nabla_Z X).
\end{equation}

Using the previously calculated connections, the components of $\mathcal{L}_X g$ are:
\begin{align*}
(\mathcal{L}_X g)(E_1, E_1) &= 2 E_1(h_1) + 2 h_3, \\
(\mathcal{L}_X g)(E_2, E_2) &= 2 E_2(h_2) - 2 h_3, \\
(\mathcal{L}_X g)(E_3, E_3) &= 2 E_3(h_3), \\
(\mathcal{L}_X g)(E_1, E_2) &= E_1(h_2) + E_2(h_1), \\
(\mathcal{L}_X g)(E_1, E_3) &= E_1(h_3) + E_3(h_1) - h_1, \\
(\mathcal{L}_X g)(E_2, E_3) &= E_2(h_3) + E_3(h_2) + h_2.
\end{align*}
For $k, j \in \{1, \dots, n\}$ and $i \in \{1, 2, 3\}$ :
\begin{align*}
(\mathcal{L}_X g)(F_k, F_j) &= F_k(f_j) + F_j(f_k), \\
(\mathcal{L}_X g)(E_i, F_k) &= E_i(f_k) + F_k(h_i).
\end{align*}
\begin{remark}\label{r1}
   Substituting $\operatorname{Ric}$, $\operatorname{Scal} = -2$, and $\mathcal{L}_X g$ into equation \eqref{eq:RY-def}, we obtain the system of PDEs:
\begin{align*}
\beta_1(0) + \frac{1}{2}\left(2E_1(h_1) + 2h_3\right) &= \lambda - 2\beta_2 \quad \iff \quad E_1(h_1) + h_3 = \lambda - 2\beta_2, \\
\beta_1(0) + \frac{1}{2}\left(2E_2(h_2) - 2h_3\right) &= \lambda - 2\beta_2 \quad \iff\quad E_2(h_2) - h_3 = \lambda - 2\beta_2, \\
\beta_1(-2) + \frac{1}{2}\left(2E_3(h_3)\right) &= \lambda - 2\beta_2 \quad \iff \quad E_3(h_3) = 2\beta_1 + \lambda - 2\beta_2, \\
\beta_1(0) + \frac{1}{2}\left(2F_k(f_k)\right) &= \lambda - 2\beta_2 \quad \iff \quad F_k(f_k) = \lambda - 2\beta_2 \\
E_1(h_2) + E_2(h_1) &= 0, \\
E_1(h_3) + E_3(h_1) - h_1 &= 0, \\
E_2(h_3) + E_3(h_2) + h_2 &= 0, \\
F_k(f_j) + F_j(f_k) &= 0 \quad (\text{pour } k \neq j), \\
E_i(f_k) + F_k(h_i) &= 0 \quad (\text{pour } i \in \{1,2,3\}, \, k \in \{1,\dots,n\}).
\end{align*}
\end{remark}
\section{Main Results}
In this section, we present our main results concerning the classification of Ricci-Yamabe solitons, Ricci solitons, Yamabe solitons, and finally gradient Ricci-Yamabe solitons on the product manifold $\Sol \times \mathbb{R}^n$.
\begin{theorem}\label{T}
  Let $X = \sum\limits_{i=1}^3 h_iE_i+\sum\limits_{k=1}^n f_k F_k$ be a smooth vector field and $(\lambda,\beta_1,\beta_2)\in\R^3$. The sextuple $(\Sol\times\R^n,g,X,\beta_1,\lambda,\beta_2)$ is a Ricci-Yamabe soliton if and only if $\lambda=2\beta_2-2\beta_1.$ Furthermore,\begin{align*}
        X(x,y,z,t_1,...,t_n)&=\left(-(2\beta_1+a_7)x-\mathfrak{a}+\sum_{k=1}^n \eta_k\left(y, t_1, \dots, \widehat{t_k}, \dots, t_n\right)\right)\partial_x\\
        &+\left(( - 2\beta_1 +a_7)y+\mathfrak{b} + \sum_{k=1}^n \varepsilon_k\left(x, t_1, \dots, \widehat{t_k}, \dots, t_n\right)\right)\partial_y+a_7\partial_z\\
      &+ \sum_{k=1}^n \left((\lambda - 2\beta_2)t_k  + \sum\limits_{\substack{1 \le j \le n \\ j \neq k}} \left(b_{jk}x+q_{jk}y+p_{jk}\right)t_j + b_kx+q_ky+p_k\right)\partial_{t_k}
    \end{align*}
  where, for all $j,k \in \{1, \dots, n\}$ with $j \ne k$, $b_{jk} = -b_{kj}$, $q_{jk} = -q_{kj}$, and $p_{jk} = -p_{kj}$ are real antisymmetry constants; $b_k$, $q_k$, $p_k$, $a_7$, $\mathfrak{a}$, and $\mathfrak{b}$ are real constants; $\eta_k$ is a family of functions depending on all variables of $\Sol \times \R^n$ except $x, z, t_k$; and $\varepsilon_k$ is also a family of functions depending on all variables of $\Sol \times \R^n$ except $y, z, t_k$, such that
   $\begin{cases}
    \sum\limits_{k=1}^n \partial_y\eta_k\left(y, t_1, \dots, \widehat{t_k}, \dots, t_n\right) = 0\\
    \sum\limits_{k=1}^n \partial_x\varepsilon_k\left(x, t_1, \dots, \widehat{t_k}, \dots, t_n\right) = 0
\end{cases}.$ 

In particular, for any functions $\mathcal{A}_k$ and $\mathcal{B}_k$ depending only on the variables $t_1, \dots, t_n$ except $t_k$, such that $\eta_k:=\mathcal{A}_k$ et $\varepsilon_k:=\mathcal{B}_k$.
\end{theorem}
\begin{proof}
   The fact that $(\Sol\times\R^n,g,X,\beta_1,\lambda,\beta_2)$ is a Ricci-Yamabe soliton is equivalent to Remark \ref{r1}. We thus have the following system of differential equations : $$\begin{cases}
        E_1(h_1) + h_3 = \lambda - 2\beta_2\quad (l_1)\\
         E_2(h_2) - h_3 = \lambda - 2\beta_2\quad (l_2) \\
        E_3(h_3) = 2\beta_1 + \lambda - 2\beta_2\quad (l_3)\\
        F_k(f_k) = \lambda - 2\beta_2 \quad (l_4)\\
E_1(h_2) + E_2(h_1) = 0\quad (l_5)\\
E_1(h_3) + E_3(h_1) - h_1 = 0\quad (l_6) \\
E_2(h_3) + E_3(h_2) + h_2 = 0\quad (l_7) \\
F_k(f_j) + F_j(f_k) = 0\quad (l_8) \quad (\text{for } k \neq j)\\
E_i(f_k) + F_k(h_i) = 0\quad (l_9) \quad (\text{for } i \in \{1,2,3\}, \, k \in \{1,\dots,n\}).
    \end{cases}$$
    \begin{itemize}
        \item Lines $(l_3)$ and $(l_4)$ give respectively \begin{equation}\label{l3}
        \partial_zh_3(x,y,z,t_1,\cdots,t_n)=2\beta_1+\lambda-2\beta_2\Leftrightarrow h_3(x,y,z,t_1,\cdots,t_n)=(2\beta_1+\lambda-2\beta_2)z+\varphi(x,y,t_1,\cdots,t_n)
    \end{equation} and {\small\begin{equation}\label{l4}
\partial_{t_k}f_k(x,y,z,t_1,\cdots,t_n)=\lambda-2\beta_2\Leftrightarrow f_k(x,y,z,t_1,\cdots,t_n)=(\lambda-2\beta_2)t_k+\xi_k(x,y,z,t_1,...,t_{k-1},\widehat{t_k},t_{k+1},...,t_n).
    \end{equation}}
    
   Substituting the equality \eqref{l4} into line $(l_8)$, we obtain:
    
$$\partial_{t_k}\xi_j(x,y,z,t_1,\dots,t_{j-1},\widehat{t_j},t_{j+1},\dots,t_n) = -\partial_{t_j}\xi_k(x,y,z,t_1,\dots,t_{k-1},\widehat{t_k},t_{k+1},\dots,t_n).$$
 We have so :
$$\partial_{t_k} \left( \partial_{t_k}\xi_j \right) = \partial_{t_k} \left( -\partial_{t_j}\xi_k \right) \iff \partial_{t_k}^2\xi_j = -\partial_{t_j} \left( \partial_{t_k}\xi_k \right)\Leftrightarrow \partial_{t_k}^2\xi_j = 0.$$
Since $\partial_{t_k}^2 \xi_j = 0$ for any time variable $t_k$ distinct from $t_j$, the function $\xi_j$ is an affine function of each of the variables $t_k$ ($k \neq j$).

By successive integration, the general form of each component $\xi_k$ is given by:
$$\xi_k(x, y, z, t_1, \dots, \widehat{t_k}, \dots, t_n) = \sum_{\substack{1 \le j \le n \\ j \neq k}} c_{jk}(x, y, z) t_j + c_k(x, y, z)$$
where the coefficients $c_{jk}(x, y, z)$ and $c_k(x, y, z)$ depend solely on the spatial variables $(x, y, z)$.

We then obtain $$f_k(x,y,z,t_1,\cdots,t_n)=(\lambda-2\beta_2)t_k+\sum_{\substack{1 \le j \le n \\ j \neq k}} c_{jk}(x, y, z) t_j + c_k(x, y, z).$$
\item The line $(l_9)$ for $i=3$ gives us the equation $\partial_z f_k + \partial_{t_k} h_3 = 0$ while
$$\partial_z f_k = \sum_{\substack{1 \le j \le n \\ j \neq k}} \frac{\partial c_{jk}}{\partial z}(x, y, z) t_j + \frac{\partial c_k}{\partial_{z}}(x, y, z)$$
$$\partial_{t_k} h_3 = \frac{\partial \varphi}{\partial t_k}(x, y, t_1, \dots, t_n)$$

The equation is then written as :
$$\sum_{\substack{1 \le j \le n \\ j \neq k}} \frac{\partial c_{jk}}{\partial z} t_j + \frac{\partial c_k}{\partial z} + \frac{\partial \varphi}{\partial t_k} = 0$$
$$\iff \frac{\partial \varphi}{\partial t_k}(x, y, t_1, \dots, t_n) = -\left( \sum_{\substack{1 \le j \le n \\ j \neq k}} \frac{\partial c_{jk}}{\partial z}(x, y, z) t_j + \frac{\partial c_k}{\partial z}(x, y, z) \right)$$

$\partial_{t_k} \varphi$ does not depend on the variable $z$. Consequently, the right-hand side must also be independent of it. Differentiating with respect to $z$ :
$$\sum_{\substack{1 \le j \le n \\ j \neq k}} \frac{\partial^2 c_{jk}}{\partial z^2}(x, y, z) t_j + \frac{\partial^2 c_k}{\partial z^2}(x, y, z) = 0$$

Since this relation holds for all values independent of the time variables $t_j$ ($j \neq k$), the coefficients of each $t_j$ as well as the constant term vanish identically.  :
$$\frac{\partial^2 c_{jk}}{\partial z^2}(x, y, z) = 0 \quad (\text{for all } j \neq k), \quad \frac{\partial^2 c_k}{\partial z^2}(x, y, z) = 0.$$
The vanishing of the second derivatives with respect to $z$ implies that the coefficients are affine functions of $z$ :
$$c_{jk}(x, y, z) = A_{jk}(x, y) z + B_{jk}(x, y)$$
$$c_k(x, y, z) = A_k(x, y) z + B_k(x, y)$$
where $A_{jk}, B_{jk}, A_k, B_k$ are arbitrary functions of $(x, y)$, with $A_{jk} = -A_{kj}$ and $B_{jk} = -B_{jk}$.
Substituting $\partial_z c_{jk} = A_{jk}(x, y)$ and $\partial_z c_k = A_k(x, y)$, we obtain :
$$\frac{\partial \varphi}{\partial t_k}(x, y, t_1, \dots, t_n) = - \sum_{\substack{1 \le j \le n \\ j \neq k}} A_{jk}(x, y) t_j - A_k(x, y)$$

By integration with respect to $t_k$, the general structure of $\varphi$ is :
$$\varphi(x, y, t_1, \dots, t_n) = - \sum_{\substack{1 \le j, k \le n \\ j < k}} A_{jk}(x, y) t_j t_k - \sum_{k=1}^n A_k(x, y) t_k + \Phi_0(x, y)$$
where $\Phi_0(x, y)$ is an arbitrary function of $(x, y)$.

Replacing the coefficients $c_{jk}$ and $c_k$ with their affine forms in $z$, we obtain:

$$f_k(x, y, z, t_1, \dots, t_n) = (\lambda - 2\beta_2)t_k + z \sum_{\substack{1 \le j \le n \\ j \neq k}} A_{jk}(x, y) t_j + \sum_{\substack{1 \le j \le n \\ j \neq k}} B_{jk}(x, y) t_j + A_k(x, y) z + B_k(x, y)$$

with the antisymmetry conditions on the spatial coefficients : 
$$A_{jk}(x, y) = -A_{kj}(x, y) \quad \text{et} \quad B_{jk}(x, y) = -B_{kj}(x, y).$$

Substituting the general expression for the function $\varphi$ obtained by integration, we get :

$$h_3(x, y, z, t_1, \dots, t_n) = (2\beta_1 + \lambda - 2\beta_2)z - \sum_{1 \le j < k \le n} A_{jk}(x, y) t_j t_k - \sum_{k=1}^n A_k(x, y) t_k + \Phi_0(x, y).$$ 
\item Lines $(l_1)$ and $(l_2)$ give respectively : 

$$e^{-z}\partial_xh_1+h_3=\lambda-2\beta_2 \iff \partial_x h_1 = e^z \left( \lambda - 2\beta_2 - h_3 \right)$$
$$e^z\partial_yh_2-h_3=\lambda-2\beta_2 \iff \partial_y h_2 = e^{-z} \left( \lambda - 2\beta_2 + h_3 \right).$$

Which yields, upon integration : 
{ \small\begin{align*}
    h_1(x, y, z, t_1, \dots, t_n)& = e^z \int \left[ (\lambda - 2\beta_2)(1 - z) - 2\beta_1 z - \Phi_0(x, y) + \sum_{1 \le j < k \le n} A_{jk}(x, y) t_j t_k + \sum_{k=1}^n A_k(x, y) t_k \right] dx \\
&+ \Psi_1(y, z, t_1, \dots, t_n)
\end{align*}}
where $\Psi_1$ is an arbitrary function independent of $x$

and
{ \small\begin{align*}
   h_2(x, y, z, t_1, \dots, t_n) &= e^{-z} \int \left[ (\lambda - 2\beta_2)(1 + z) + 2\beta_1 z + \Phi_0(x, y) - \sum_{1 \le j < k \le n} A_{jk}(x, y) t_j t_k - \sum_{k=1}^n A_k(x, y) t_k \right] dy\\
&+ \Psi_2(x, z, t_1, \dots, t_n)
\end{align*}}
where $\Psi_2$ is an arbitrary function independent of $y$.

\item The line $(l_9)$ for $i\in\{1;2\}$ becomes $$\begin{cases}
    e^{-z}\partial_xf_k+\partial_{t_k}h_1=0\\
   &  \quad k\in\{1,..,n\} \\
    e^z\partial_yf_k+\partial_{t_k}h_2=0
\end{cases}.$$
We have : \begin{equation}\label{1}
     \partial_x f_k = z \sum_{\substack{1 \le j \le n \\ j \neq k}} \partial_x A_{jk}(x, y) t_j + \sum_{\substack{1 \le j \le n \\ j \neq k}} \partial_x B_{jk}(x, y) t_j + z \partial_x A_k(x, y) + \partial_x B_k(x, y),
\end{equation}
\begin{equation}\label{2}
     \partial_y f_k = z \sum_{\substack{1 \le j \le n \\ j \neq k}} \partial_y A_{jk}(x, y) t_j + \sum_{\substack{1 \le j \le n \\ j \neq k}} \partial_y B_{jk}(x, y) t_j + z \partial_y A_k(x, y) + \partial_y B_k(x, y),
\end{equation}
\begin{equation}\label{3}
     \partial_{t_k}h_1 = e^z \int \left[ \sum_{\substack{1 \le j \le n \\ j \neq k}} A_{jk}(x, y) t_j + A_k(x, y) \right] dx + \partial_{t_k}\Psi_1(y, z, t_1, \dots, t_n),
\end{equation}

\begin{equation}\label{4}
    \partial_{t_k}h_2 = e^{-z} \int \left[ -\sum_{\substack{1 \le j \le n \\ j \neq k}} A_{jk}(x, y) t_j - A_k(x, y) \right] dy + \partial_{t_k}\Psi_2(x, z, t_1, \dots, t_n).
\end{equation}
Using \eqref{1} and \eqref{3} gives :\begin{align*}
    & z \sum_{\substack{1 \le j \le n \\ j \neq k}} \partial_x A_{jk}(x, y) t_j + \sum_{\substack{1 \le j \le n \\ j \neq k}} \partial_x B_{jk}(x, y) t_j + z \partial_x A_k(x, y) + \partial_x B_k(x, y)\\
    &+ e^{2z} \int \left[ \sum_{\substack{1 \le j \le n \\ j \neq k}} A_{jk}(x, y) t_j + A_k(x, y) \right] dx =-e^z \partial_{t_k}\Psi_1(y, z, t_1, \dots, t_n)
\end{align*}
Given that the function $(y,z,t_1,..,t_n)\mapsto -e^z \partial_{t_k}\Psi_1(y, z, t_1, \dots, t_n)$ does not depend on the variable $x$, for the equality to hold, it is necessary and sufficient that the function \begin{align*}
\mathcal{F}:\Sol\times \R^n\to&\R\\
    (x,y,z,t_1,...,t_n)\mapsto&  z \sum_{\substack{1 \le j \le n \\ j \neq k}} \partial_x A_{jk}(x, y) t_j + \sum_{\substack{1 \le j \le n \\ j \neq k}} \partial_x B_{jk}(x, y) t_j + z \partial_x A_k(x, y) + \partial_x B_k(x, y)\\
    &+ e^{2z} \int \left[ \sum_{\substack{1 \le j \le n \\ j \neq k}} A_{jk}(x, y) t_j + A_k(x, y) \right] dx
\end{align*} does not depend on the variable $x$. This proves that the functions $A_{jk}$, $A_k$, $\partial_xB_{jk}$, and $\partial_xB_k$ will depend only on the variable $y$, because if they depended on the variable $x$, then the function\begin{align*}
    \mathcal{E}:\Sol\times \R^n\to&\R\\
(x,y,z,t_1,..,t_n)\mapsto& e^{2z} \int \left[ \sum_{\substack{1 \le j \le n \\ j \neq k}} A_{jk}(x, y) t_j + A_k(x, y) \right] dx
\end{align*} will depend on the variable $x$, and since the function $$(x,y,z,t_1,...,t_n)\mapsto  z \sum_{\substack{1 \le j \le n \\ j \neq k}} \partial_x A_{jk}(x, y) t_j + \sum_{\substack{1 \le j \le n \\ j \neq k}} \partial_x B_{jk}(x, y) t_j + z \partial_x A_k(x, y) + \partial_x B_k(x, y)$$ is affine with respect to the variable $z$, so it cannot compensate for the function $\mathcal{E}$ in such a way that the function $\mathcal{F}$ is independent of the variable $x$.

Using \eqref{2} and \eqref{4} yields :
\begin{align*}
   & z \sum_{\substack{1 \le j \le n \\ j \neq k}} \partial_y A_{jk}(x, y) t_j + \sum_{\substack{1 \le j \le n \\ j \neq k}} \partial_y B_{jk}(x, y) t_j + z \partial_y A_k(x, y) + \partial_y B_k(x, y)\\
   & -e^{-2z} \int \left[ \sum_{\substack{1 \le j \le n \\ j \neq k}} A_{jk}(x, y) t_j + A_k(x, y) \right] dy =e^{-z}\partial_{t_k}\Psi_2(x, z, t_1, \dots, t_n).
\end{align*}
Using the same reasoning as before, we can conclude that the equation is true only if and only if the functions $A_{jk}$, $A_k$, $\partial_yB_{jk}$, and $\partial_yB_k$ depend only on the variable $x$. This leads us to conclude that $A_{jk}(x,y)=0$, $A_k(x,y)=0$, $B_{jk}(x,y)=b_{jk}x+q_{jk}y+p_{jk}$ et $B_k(x,y)=b_kx+q_ky+p_k$.
with   $b_{jk} = -b_{kj}$ , $q_{jk} = -q_{kj}$, $p_{jk} = -p_{kj}$ real antisymmetry constants, and $b_k$, $q_k$, $p_k$ real constants.

This yields : $$ \left(\sum_{\substack{1 \le j \le n \\ j \neq k}} b_{jk} t_j  +b_k\right)e^{-z} =-\partial_{t_k}\Psi_1(y, z, t_1, \dots, t_n)$$
$$ \Psi_1(y, z, t_1, \dots, t_n)=-\sum_{k=1}^n\left(\left(\sum_{\substack{1 \le j \le n \\ j \neq k}}\left( b_{jk} t_j  +b_k\right)\right)e^{-z}t_k+\psi_k(y,z,t_1,...,\widehat{t_k},...,t_n)\right).$$
$$\left(\sum_{\substack{1 \le j \le n \\ j \neq k}} q_{jk}t_j  +q_k\right)e^z  =\partial_{t_k}\Psi_2(x, z, t_1, \dots, t_n)$$
$$\Psi_2(x,z,t_1,\dots,t_n)=\sum_{k=1}^n\left(\left(\sum_{\substack{1 \le j \le n \\ j \neq k}} \left(q_{jk}t_j  +q_k\right)\right)e^zt_k+\zeta_k(x,z,t_1,...,\widehat{t_k},...,t_n)\right).$$
{\small$$
   f_k(x, y, z, t_1, \dots, t_n) = (\lambda - 2\beta_2)t_k  + \sum\limits_{\substack{1 \le j \le n \\ j \neq k}} \left(b_{jk}x+q_{jk}y+p_{jk}\right)t_j + b_kx+q_ky+p_k,$$}
{\small\begin{align*}
    h_1(x, y, z, t_1, \dots, t_n)& = \left((\lambda - 2\beta_2)(1 - z) - 2\beta_1 z\right)xe^z-e^z \int  \Phi_0(x, y) dx \\
&-\sum_{k=1}^n\left(\left(\sum_{\substack{1 \le j \le n \\ j \neq k}}\left( b_{jk} t_j  +b_k\right)\right)e^{-z}\right)t_k+\sum_{k=1}^n\psi_k(y,z,t_1,...,\widehat{t_k},...,t_n)
\end{align*}}

{ \small\begin{align*}
   h_2(x, y, z, t_1, \dots, t_n) &= \left((\lambda - 2\beta_2)(1 + z) + 2\beta_1 z \right)ye^{-z} +e^{-z} \int \Phi_0(x, y) dy\\
&+\sum_{k=1}^n\left(\left(\sum_{\substack{1 \le j \le n \\ j \neq k}} \left(q_{jk}t_j  +q_k\right)\right)e^z\right)t_k+\sum_{k=1}^n\zeta_k(x,z,t_1,...,\widehat{t_k},...,t_n)
\end{align*}}

$$h_3(x, y, z, t_1, \dots, t_n) = (2\beta_1 + \lambda - 2\beta_2)z + \Phi_0(x, y).$$ 
\item The line $(l_6)$ gives  :{\small \begin{align*}
  &  e^{-z}\partial_xh_3 + \partial_zh_1 - h_1 = 0 \Leftrightarrow  e^{-z}\partial_x \Phi_0(x, y)+2\sum_{k=1}^n\left(\left(\sum_{\substack{1 \le j \le n \\ j \neq k}}\left( b_{jk} t_j  +b_k\right)\right)e^{-z}\right)t_k\\
&+\left(-\lambda +2\beta_2 - 2\beta_1 \right)xe^z+\sum_{k=1}^n\left(\partial_z\psi_k(y,z,t_1,...,\widehat{t_k},...,t_n)-\psi_k(y,z,t_1,...,\widehat{t_k},...,t_n)\right)=0
\end{align*}} the equation is true if {\small$$\begin{cases}
\lambda =2\beta_2 - 2\beta_1\\ 
    \Phi_0(x, y)=a_1(y)x+a_2(y)\\
     \sum\limits_{k=1}^n\left(\partial_z\psi_k(y,z,t_1,...,\widehat{t_k},...,t_n)-\psi_k(y,z,t_1,...,\widehat{t_k},...,t_n)\right)=-a_1(y) e^{-z}-2\sum\limits_{k=1}^n\left(\left(\sum\limits_{\substack{1 \le j \le n \\ j \neq k}}\left( b_{jk} t_j  +b_k\right)\right)e^{-z}t_k\right)
\end{cases}$$}
This partial differential equation involves a summation over the index $k$ of first-order terms in $z$. By natural decomposition (by term-by-term identification in the sum), each component $\psi_k$ satisfies a first-order linear differential equation with respect to $z$.

For a fixed index $k$ ($k \in \{1, \dots, n\}$), the equation reduces to :
$$\partial_z\psi_k(y,z,t_1,...,\widehat{t_k},...,t_n) - \psi_k(y,z,t_1,...,\widehat{t_k},...,t_n) = -\frac{a_1(y) e^{-z}}{n} - 2\left(\sum_{\substack{1 \le j \le n \\ j \neq k}}\left( b_{jk} t_j  +b_k\right)\right) t_k e^{-z}.$$
By multiplying by the integrating factor $I(z) = e^{-z}$ :
$$e^{-z}\partial_z\psi_k - e^{-z}\psi_k = -\frac{a_1(y)}{n}e^{-2z} - 2\left(\sum_{\substack{1 \le j \le n \\ j \neq k}}\left( b_{jk} t_j  +b_k\right)\right) t_k e^{-2z}$$

The left-hand side is expressed as the derivative of the product :
$$\frac{\partial}{\partial z}\left(e^{-z}\psi_k\right) = -\frac{a_1(y)}{n}e^{-2z} - 2\left(\sum_{\substack{1 \le j \le n \\ j \neq k}}\left( b_{jk} t_j  +b_k\right)\right) t_k e^{-2z}.$$

By integrating each term with respect to $z$ :
$$e^{-z}\psi_k = \frac{a_1(y)}{2n}e^{-2z} + \left(\sum_{\substack{1 \le j \le n \\ j \neq k}}\left( b_{jk} t_j  +b_k\right)\right) t_k e^{-2z} + \eta_k\left(y, t_1, \dots, \widehat{t_k}, \dots, t_n\right).$$
For this reason $\psi_k$ :
$$\psi_k(y,z,t_1,...,\widehat{t_k},...,t_n) = \frac{a_1(y)}{2n} e^{-z} + \left(\sum_{\substack{1 \le j \le n \\ j \neq k}}\left( b_{jk} t_j  +b_k\right)\right) t_k e^{-z} + \eta_k\left(y, t_1, \dots, \widehat{t_k}, \dots, t_n\right)e^z$$

where $\eta_k\left(y, t_1, \dots, \widehat{t_k}, \dots, t_n\right)$ is an arbitrary function associated with the $k$-th direction, independent of $z$ and $t_k$.
This yields {\small\begin{align*}
    h_1(x, y, z, t_1, \dots, t_n)& = \left(\lambda - 2\beta_2\right)xe^z- \left(\frac{a_1(y)}{2}x^2+a_2(y)x+a_3(y)\right)e^z\\
&-\sum_{k=1}^n\left(\left(\sum_{\substack{1 \le j \le n \\ j \neq k}}\left( b_{jk} t_j  +b_k\right)\right)e^{-z}\right)t_k+\frac{a_1(y)}{2} e^{-z}\\
&+\sum_{k=1}^n \left( \left(\sum_{\substack{1 \le j \le n \\ j \neq k}}\left( b_{jk} t_j  +b_k\right)\right) t_k e^{-z} + \eta_k\left(y, t_1, \dots, \widehat{t_k}, \dots, t_n\right)e^z\right).
\end{align*}}
{\small\begin{align*}
    h_1(x, y, z, t_1, \dots, t_n)& = \left(\lambda - 2\beta_2\right)xe^z- \left(\frac{a_1(y)}{2}x^2+a_2(y)x+a_3(y)\right)e^z
+\frac{a_1(y)}{2} e^{-z}\\
&+\sum_{k=1}^n \eta_k\left(y, t_1, \dots, \widehat{t_k}, \dots, t_n\right)e^z.
\end{align*}}
\item Line $(l_7)$ gives : {\small\begin{align*}
    e^z\partial_yh_3 + \partial_zh_2 + h_2 = 0&\Leftrightarrow e^z(a_1'(y)x+a'_2(y))+2\sum_{k=1}^n\left(\left(\sum_{\substack{1 \le j \le n \\ j \neq k}} \left(q_{jk}t_j  +q_k\right)\right)e^z\right)t_k\\
    &+\sum_{k=1}^n\partial_z\zeta_k(x,z,t_1,...,\widehat{t_k},...,t_n)+\sum_{k=1}^n\zeta_k(x,z,t_1,...,\widehat{t_k},...,t_n)=0
\end{align*}}
this equation is true only when $$\begin{cases}
    a_1(y)=a_4y+a_5,\quad a_2(y)= a_6y+a_7\\
    \sum\limits_{k=1}^n\partial_z\zeta_k(x,z,t_1,...,\widehat{t_k},...,t_n)+\sum\limits_{k=1}^n\zeta_k(x,z,t_1,...,\widehat{t_k},...,t_n)&=-e^z(a_4x+a_6)\\
    &-2\sum\limits_{k=1}^n\left(\left(\sum\limits_{\substack{1 \le j \le n \\ j \neq k}} \left(q_{jk}t_j  +q_k\right)\right)e^z\right)t_k
\end{cases}$$

By proceeding with a term-by-term decomposition, each component $\zeta_k(x, z, t_1, \dots, \widehat{t_k}, \dots, t_n)$ satisfies a first-order linear differential equation with respect to $z$ :

$$\partial_z\zeta_k(x, z, t_1, \dots, \widehat{t_k}, \dots, t_n) + \zeta_k(x, z, t_1, \dots, \widehat{t_k}, \dots, t_n) = -\frac{a_4x+a_6}{n}e^z - 2\left(\sum_{\substack{1 \le j \le n \\ j \neq k}} \left(q_{jk}t_j  +q_k\right)\right) t_k e^z.$$
Since the coefficient of $\zeta_k$ is $+1$, the integrating factor is $I(z) = e^{\int 1 \, dz} = e^z$. Multiplying the equation by $e^z$ :
$$e^z\partial_z\zeta_k + e^z\zeta_k = -\frac{a_4x+a_6}{n}e^{2z} - 2\left(\sum_{\substack{1 \le j \le n \\ j \neq k}} \left(q_{jk}t_j  +q_k\right)\right) t_k e^{2z}.$$
The left-hand side can be rewritten as the total derivative of the product :
$$\frac{\partial}{\partial z}\left(e^z\zeta_k\right) = -\frac{a_4x+a_6}{n}e^{2z} - 2\left(\sum_{\substack{1 \le j \le n \\ j \neq k}} \left(q_{jk}t_j  +q_k\right)\right) t_k e^{2z}.$$
So
$$e^z\zeta_k = -\left(\frac{a_4x+a_6}{n}\right)\left(\frac{e^{2z}}{2}\right) - 2\left(\sum_{\substack{1 \le j \le n \\ j \neq k}} \left(q_{jk}t_j  +q_k\right)\right) t_k \left(\frac{e^{2z}}{2}\right) + \varepsilon_k\left(x, t_1, \dots, \widehat{t_k}, \dots, t_n\right)$$

$$e^z\zeta_k = -\left(\frac{a_4x+a_6}{2n}\right)e^{2z} - \left(\sum_{\substack{1 \le j \le n \\ j \neq k}} \left(q_{jk}t_j  +q_k\right)\right) t_k e^{2z} + \varepsilon_k\left(x, t_1, \dots, \widehat{t_k}, \dots, t_n\right).$$
By multiplying the entire expression by $e^{-z}$, one obtains the explicit expression for the component $\zeta_k$ :

$$\zeta_k(x, z, t_1, \dots, \widehat{t_k}, \dots, t_n) = -\left(\frac{a_4x+a_6}{2n}\right)e^z - \left(\sum_{\substack{1 \le j \le n \\ j \neq k}} \left(q_{jk}t_j  +q_k\right)\right) t_k e^z + \varepsilon_k\left(x, t_1, \dots, \widehat{t_k}, \dots, t_n\right)e^{-z}$$

where $\varepsilon_k$ is an arbitrary function independent of $z$ and $t_k$.

This yields
{ \small\begin{align*}
   h_2(x, y, z, t_1, \dots, t_n) &=  - 2\beta_1ye^{-z} +e^{-z} \left(\left(\frac{a_4y^2}{2}+a_5y\right)x+\frac{a_6}{2}y^2+a_7y+a_8(x)\right)\\
   &+\sum_{k=1}^n\left(\left(\sum_{\substack{1 \le j \le n \\ j \neq k}} \left(q_{jk}t_j  +q_k\right)\right)e^z\right)t_k-\left(\frac{a_4x+a_6}{2}\right)e^z\\
& - \sum_{k=1}^n \left[\left(\sum_{\substack{1 \le j \le n \\ j \neq k}} \left(q_{jk}t_j  +q_k\right)\right) t_k\right] e^z + \sum_{k=1}^n \varepsilon_k\left(x, t_1, \dots, \widehat{t_k}, \dots, t_n\right)e^{-z}.
\end{align*}}
{ \small\begin{align*}
   h_2(x, y, z, t_1, \dots, t_n) &=  - 2\beta_1ye^{-z} +e^{-z} \left(\left(\frac{a_4y^2}{2}+a_5y\right)x+\frac{a_6}{2}y^2+a_7y+a_8(x)\right)\\
   &-\left(\frac{a_4x+a_6}{2}\right)e^z + \sum_{k=1}^n \varepsilon_k\left(x, t_1, \dots, \widehat{t_k}, \dots, t_n\right)e^{-z}.
\end{align*}}
$h_1$ becomes again: {\small\begin{align*}
    h_1(x, y, z, t_1, \dots, t_n)& = - 2\beta_1xe^z- \left(\left(\frac{a_4y+a_5}{2}\right)x^2+(a_6y+a_7)x+a_3(y)\right)e^z\\
&+\left(\frac{a_4y+a_5}{2}\right) e^{-z}+\sum_{k=1}^n \eta_k\left(y, t_1, \dots, \widehat{t_k}, \dots, t_n\right)e^z.
\end{align*}}
\item Finally, the line $(l_5)$ yields : $e^{-z}\partial_xh_2+ e^z\partial_yh_1 = 0$ \begin{align*}
    &e^{-z}\left(\frac{a_4}{2}y^2+a_5y+a'_8(x)-\frac{a_4}{2}e^z+ \sum_{k=1}^n \partial_x\varepsilon_k\left(x, t_1, \dots, \widehat{t_k}, \dots, t_n\right)e^{-z}\right)\\
    &+e^z\left(-\left(\frac{a_4}{2}x^2+a_6x+a'_3(y)\right)+\frac{a_4}{2}e^{-z}+\sum_{k=1}^n \partial_y\eta_k\left(y, t_1, \dots, \widehat{t_k}, \dots, t_n\right)e^z\right)=0
\end{align*}
By expanding and grouping the expression according to the powers of $e^z$ and $e^{-z}$, we obtain :

\begin{align*}
    &\left(\sum_{k=1}^n \partial_y\eta_k\left(y, t_1, \dots, \widehat{t_k}, \dots, t_n\right)\right) e^{2z} 
    - \left(\frac{a_4}{2}x^2 + a_6x + a'_3(y)\right) e^z 
    + \left(-\frac{a_4}{2} + \frac{a_4}{2}\right) \\
    &+ \left(\frac{a_4}{2}y^2 + a_5y + a'_8(x)\right) e^{-z} 
    + \left(\sum_{k=1}^n \partial_x\varepsilon_k\left(x, t_1, \dots, \widehat{t_k}, \dots, t_n\right)\right) e^{-2z} = 0.
\end{align*}
Since this expression must hold identically for all $z$, the coefficients of each linearly independent exponential function of $z$ ($e^{2z}$, $e^z$, $e^{-z}$, and $e^{-2z}$) must vanish independently.

$$\begin{cases}
    \sum\limits_{k=1}^n \partial_y\eta_k\left(y, t_1, \dots, \widehat{t_k}, \dots, t_n\right) = 0\\
    \frac{a_4}{2}x^2 + a_6x + a'_3(y) = 0\\
    \frac{a_4}{2}y^2 + a_5y + a'_8(x) = 0\\
    \sum\limits_{k=1}^n \partial_x\varepsilon_k\left(x, t_1, \dots, \widehat{t_k}, \dots, t_n\right) = 0
\end{cases}.$$

The parameters therefore reduce to :
$$a_4 = 0, \quad a_5 = 0, \quad a_6 = 0, \quad a_3(y) = \mathfrak{a}\in\R, \quad a_8(x) = \mathfrak{b}\in\R.$$
Finaly 
{\small$$ \begin{cases}
    h_1(x, y, z, t_1, \dots, t_n) =  - 2\beta_1xe^z- \left(a_7x+\mathfrak{a}\right)e^z+\sum\limits_{k=1}^n \eta_k\left(y, t_1, \dots, \widehat{t_k}, \dots, t_n\right)e^z,\\
h_2(x, y, z, t_1, \dots, t_n) =  - 2\beta_1ye^{-z} +e^{-z} \left(a_7y+\mathfrak{b}\right) + \sum\limits_{k=1}^n \varepsilon_k\left(x, t_1, \dots, \widehat{t_k}, \dots, t_n\right)e^{-z}\\
h_3(x, y, z, t_1, \dots, t_n) =a_7,\\ 
   f_k(x, y, z, t_1, \dots, t_n) = (\lambda - 2\beta_2)t_k  + \sum\limits_{\substack{1 \le j \le n \\ j \neq k}} \left(b_{jk}x+q_{jk}y+p_{jk}\right)t_j + b_kx+q_ky+p_k,
\end{cases}$$}
 \end{itemize}
\end{proof}
\begin{corollary}
   Under the hypotheses of Theorem \ref{T}, when \begin{itemize}
        \item $(\beta_1,\beta_2)=(1,0)$, then $(\Sol\times\R^n,g,X,\lambda)$ is an expanding Ricci soliton;
        \item $(\beta_1,\beta_2)=(0,1)$, then $(\Sol\times\R^n,g,X,\lambda)$ is a shrinking Yamabe soliton.
    \end{itemize}
\end{corollary}
\begin{corollary}
    Let $f$ be a smooth function on $\Sol \times \R^n$. The sextuple $(\Sol \times \R^n, g, \nabla f, \beta_1, \lambda, \beta_2)$ is a gradient Ricci-Yamabe soliton if and only if \begin{equation}
        f(x,y,z,t_1, \dots, t_n) = \frac{\lambda - 2\beta_2}{2}\sum_{k=1}^n t_k^2 + \sum_{k=1}^n \sum_{j < k} p_{jk} t_j t_k + \sum_{k=1}^n p_k t_k + C
    \end{equation}
 where $C$ is a real constant and $p_{jk}, p_k$ are constants given by Theorem \ref{T}.
\end{corollary}
\begin{proof}
   Suppose there exists a smooth function $f$ on $\Sol \times \R^n$ such that $(\Sol \times \R^n, g, \nabla f, \beta_1, \lambda, \beta_2)$ is a gradient Ricci-Yamabe soliton. This yields $$\nabla f=\sum_{i=1}^n\sum_{j=1}^n\left(g^{ij}\partial_jf\right)\partial_i\Leftrightarrow \begin{cases}
        e^{-z}\partial_xf=-(2\beta_1+a_7)x-\mathfrak{a}+\sum\limits_{k=1}^n \eta_k\left(y, t_1, \dots, \widehat{t_k}, \dots, t_n\right),\\
        e^z\partial_yf=( - 2\beta_1 +a_7)y+\mathfrak{b} + \sum_{k=1}^n \varepsilon_k\left(x, t_1, \dots, \widehat{t_k}, \dots, t_n\right),\\
        \partial_z f=a_7,\\
        \partial_{t_k}f=(\lambda - 2\beta_2)t_k  + \sum\limits_{\substack{1 \le j \le n \\ j \neq k}} \left(b_{jk}x+q_{jk}y+p_{jk}\right)t_j + b_kx+q_ky+p_k.
    \end{cases}$$  

   We have : 
$$\partial_z f = a_7 \iff f(x, y, z, t_1, \dots, t_n) = a_7 z + \varphi(x, y, t_1, \dots, t_n)$$
so   $\partial_x f = \partial_x \varphi$ and $\partial_y f = \partial_y \varphi$. By substituting into the first two equations :
$$\partial_x \varphi= e^z \left[ -(2\beta_1+a_7)x - \mathfrak{a} + \sum_{k=1}^n \eta_k \right]$$
$$\partial_y \varphi = e^{-z} \left[ (-2\beta_1 + a_7)y + \mathfrak{b} + \sum_{k=1}^n \varepsilon_k \right]$$

Since $\varphi$ does not depend on $z$, its derivatives $\partial_x \varphi$ and $\partial_y \varphi$ are independent of $z$. For these equalities to hold for all $z$, the coefficients multiplying $e^z$ and $e^{-z}$ must vanish identically :
\begin{itemize}
    \item $2\beta_1 + a_7 = 0$
    \item $\mathfrak{a} = 0$
    \item $\sum\limits_{k=1}^n \eta_k = 0$
    \item $-2\beta_1 + a_7 = 0$
    \item $\mathfrak{b} = 0$
    \item $\sum\limits_{k=1}^n \varepsilon_k = 0$
\end{itemize}

By combining $2\beta_1 + a_7 = 0$ and $-2\beta_1 + a_7 = 0$, we obtain :
$$a_7 = 0 \quad \text{and} \quad \beta_1 = 0$$

Consequently, $\partial_x \varphi = 0$ and $\partial_y \varphi = 0$, which means that $\varphi$ depends on neither $x$ nor $y$. The function $f$ thus reduces to a function of the variables $t_k$ :
$$f(x, y, z, t_1, \dots, t_n) = h(t_1, \dots, t_n).$$

The last equation then becomes :
$$\partial_{t_k} h = (\lambda - 2\beta_2)t_k + \sum_{\substack{1 \le j \le n \\ j \neq k}} \left(b_{jk}x + q_{jk}y + p_{jk}\right)t_j + b_k x + q_k y + p_k$$

Since the left-hand side ($\partial_{t_k} h$) depends on neither $x$ nor $y$, the coefficients of $x$ and $y$ on the right-hand side must be zero for all $k$ :
$$b_{jk} = 0, \quad q_{jk} = 0, \quad b_k = 0, \quad q_k = 0$$

The system simplifies for each component $t_k$ by :
$$\partial_{t_k} h(t_1, \dots, t_n) = (\lambda - 2\beta_2)t_k + \sum_{\substack{1 \le j \le n \\ j \neq k}} p_{jk}t_j + p_k.$$

Integrating this system for $h$ yields the final expression for $f$ :
$$f(x,y,z,t_1, \dots, t_n) = \frac{\lambda - 2\beta_2}{2}\sum_{k=1}^n t_k^2 + \sum_{k=1}^n \sum_{j < k} p_{jk} t_j t_k + \sum_{k=1}^n p_k t_k + C$$

subject to compatibility constraints :
$$a_7 = 0, \quad \beta_1 = 0, \quad \mathfrak{a} = 0, \quad \mathfrak{b} = 0, \quad \sum_{k=1}^n \eta_k = 0, \quad \sum_{k=1}^n \varepsilon_k = 0.$$
\end{proof}


\begin{thebibliography}{99}

\bibitem{Besse}
A.~L.~Besse,
\newblock \emph{Einstein Manifolds},
\newblock Springer-Verlag, Berlin, 1987.
\bibitem{BoussoNdiayeJDSGT2025}
A.~Bousso and A.~Ndiaye,
\newblock $\eta$-Ricci-Bourguignon soliton on the hyperbolic spaces,
\newblock \emph{Journal of Dynamical Systems and Geometric Theories}
(2025).
\newblock DOI:
\doi{10.47974/JDSGT-2025-09003}.
\bibitem{BoussoNdiaye2026Hn}
A.~Bousso and A.~Ndiaye,
\newblock Ricci solitons on the Poincaré upper half plane,
\newblock \emph{Honam Mathematical Journal}
\textbf{48} (2026), no.~2, 329--348.
\newblock DOI:
\doi{10.5831/HMJ.2026.48.2.329}.

\bibitem{BoussoNdiayeH2R2025}
A.~Bousso and A.~Ndiaye,
\newblock $h$-Ricci-Bourguignon solitons on the
$\mathbb{H}^2\times\mathbb{R}$ Lie group,
\newblock \emph{Global Journal of Advanced Research on
Classical and Modern Geometries}
\textbf{14} (2025), no.~2, 200--207.

\bibitem{BoussoNdiayeSol3}
A.~Bousso and A.~Ndiaye,
\newblock Ricci--Yamabe solitons on the Lie group
$\mathrm{Sol}_3$,
\newblock \emph{Journal of Universal Mathematics}
\textbf{9} (2026), no.~1, 36--45.
\newblock DOI:
\doi{10.33773/jum.1906429}.

\bibitem{DiopBoussoNdiayeMandal2026}
M.~N.~Diop, A.~Bousso, A.~Ndiaye and A.~Mandal,
\newblock $(h,\eta)$-Ricci-Bourguignon soliton on the
Poincaré Disk $D^2$,
\newblock \emph{Konuralp Journal of Mathematics}
\textbf{14} (2026), no.~1.

\bibitem{doCarmo}
M.~P.~do Carmo,
\newblock \emph{Riemannian Geometry},
\newblock Birkhäuser, Boston, 1992.

\bibitem{Hamilton}
R.~S.~Hamilton,
\newblock The Ricci flow on surfaces,
\newblock \emph{Mathematics and General Relativity},
\newblock Contemporary Mathematics \textbf{71} (1988), 237--262.

\bibitem{Lee}
J.~M.~Lee,
\newblock \emph{Riemannian Manifolds: An Introduction to Curvature},
\newblock Graduate Texts in Mathematics, Vol.~176,
\newblock Springer, New York, 1997.

\bibitem{ONeill}
B.~O'Neill,
\newblock \emph{Semi-Riemannian Geometry with Applications to Relativity},
\newblock Academic Press, New York, 1983.

\bibitem{Petersen}
P.~Petersen,
\newblock \emph{Riemannian Geometry},
\newblock 3rd ed.,
\newblock Springer, New York, 2016.

\end{thebibliography}
\end{document}